\documentclass[11pt]{article}
\usepackage[hmargin=1.0in]{geometry}
\usepackage{graphicx}
\usepackage[T1]{fontenc}
\usepackage[sc]{mathpazo}
\usepackage{amsmath,amsthm,amssymb,amsfonts}
\usepackage{upgreek}
\usepackage[english]{babel}
\usepackage{color}
\usepackage{hyperref}
\usepackage{enumerate}

\newtheorem{thm}{Theorem}[section]

\newtheorem{lem}{Lemma}[section]
\newtheorem{prop}{Proposition}[section]

\newcommand{\C}{\mathcal{C}}
\newcommand{\Cexp}{\mathcal{C}_{\mathrm{exp}}}
\newcommand{\Cimp}{\mathcal{C}_{\mathrm{imp}}}
\newcommand{\hFE}{h_\textup{FE}}
\newcommand{\Real}{\mathbb R}
\newcommand{\Comp}{\mathbb C}
\newcommand{\transp}{\top}
\DeclareMathOperator*{\rank}{rank}
\renewcommand{\Re}{\operatorname{Re}}
\renewcommand{\Im}{\operatorname{Im}}
\providecommand{\abs}[1]{\lvert#1\rvert}

\providecommand{\conj}[1]{\bar{#1}}

\providecommand{\floorM}[1]{\left\lfloor#1\right\rfloor}
\providecommand{\diff}{\mathop{}\!\mathrm{d}}

\title{Existence and optimality of strong stability preserving linear multistep methods: a duality-based approach}
\author{Adri{\'a}n N{\'e}meth \and David I. Ketcheson\footnote{Author email addresses: \texttt{nemetha@sze.hu}, \texttt{david.ketcheson@kaust.edu.sa}.
This work was supported by Award No.\ FIC/2010/05 - 2000000231,
made by King Abdullah University of Science and Technology; and by
T\'AMOP-4.2.2.A-11/1/KONV-2012-0012:
Basic research for the development of hybrid and electric vehicles - The Project is
supported by the Hungarian Government and co-financed by the European Social Fund.
}}
\begin{document}
\maketitle

\begin{abstract}
We prove the existence of explicit linear multistep methods of 
any order with positive coefficients.  Our approach is based on
formulating a linear programming problem and establishing infeasibility
of the dual problem.  This yields a number of other theoretical 
advances.
\end{abstract}

\section{Introduction}
In this work we study numerical methods for the solution of the initial
value problem (IVP)
\begin{align} \label{ivp}
    u'(t) & = f(t,u) & u(t_0) & = u_0 & t_0\le t \le T,
\end{align}
under the assumption that the solution is monotone in time:
\begin{align} \label{mono-cont}
    \|u(t+h)\| & \le \|u(t)\| & \forall h \ge 0.
\end{align}
Here $u : \Real \to \Real^m$, and $\|\cdot\|$ is any convex functional.  The theory pursued herein is also
relevant when $u$ satisfies a contractivity or positivity condition.  We will usually
write $f(u)$ instead of $f(t,u)$ merely to keep the notation simpler.

We focus on the class of methods for solving \eqref{ivp} that are known as
linear multistep methods (LMMs).  To solve \eqref{ivp} by a linear multistep
method, we define a sequence of times $t_0 < t_1 < \ldots < t_N = T$ where $t_n := t_{n-1} + h$,
and compute the values $u_n\approx u(t_n)$ sequentially by
\begin{align} \label{lmm}
  u_n = \sum_{j=1}^{k} \alpha_j u_{n-j} + h \sum_{j=0}^k \beta_j f(u_{n-j}).
\end{align}
Some prescription must be given for the starting
values $u_0, \dots, u_{k-1}$.
If $\beta_0=0$, the method is said to be explicit; otherwise it is implicit.

We are interested in methods that preserve a discrete version of the monotonicity
condition \eqref{mono-cont}; namely
\begin{align} \label{monotonicity}
    \|u_n\| \le \max\left\{ \|u_{n-1}\|,\dots,\|u_{n-k}\| \right\}.
\end{align}
The backward Euler method achieves \eqref{monotonicity} under any step
size, as long as $f$ is such that \eqref{mono-cont} is satisfied.

In order to achieve the discrete monotonicity property \eqref{monotonicity} with explicit methods, or with
implicit methods of order higher than one, we assume a
stronger condition than \eqref{mono-cont}.  We require that $f$ be monotone
under a forward Euler step, with some restriction on the step size:
\begin{align} \label{forward-Euler}
    \|u + h f(u)\| & \le \|u\| & \forall 0 \le h \le \hFE(u).
\end{align}
Under assumption \eqref{forward-Euler} it can be shown that any method
\eqref{lmm} with non-negative coefficients preserves discrete monotonicity
\eqref{monotonicity} under some time step restriction.  The step size
restriction that guarantees monotonicity is then
\begin{align*}
    0 \le h \le \C \hFE,
\end{align*}
where the factor
\begin{align}
  \C :=
  \begin{cases}
    \min\limits_{\substack{1 \le j \le k\\ \beta_j \ne 0}} \alpha_j/\beta_j &
    \text{ if } \alpha_j \ge 0 \text{ for } 1 \le j \le k \text{ and }
    \beta_j \ge 0 \text{ for } 0 \le j \le k,\\
    0 & \text{ otherwise}
  \end{cases}
\end{align}
is known as the {\em threshold factor} or {\em strong stability preserving
coefficient} (SSP coefficient) of the method.

Runge--Kutta methods with positive SSP coefficient are subject to restrictive
order barriers: explicit methods have order at most four and implicit methods
have order at most six.  SSP linear multistep methods of high order are therefore
of particular interest.
The objective of the present work is to investigate existence of methods
with $\C>0$ and bounds on the value of $\C$ for methods with a prescribed
order of accuracy and number of steps.

In this paper we have assumed a fixed step size;
the case of SSP explicit linear multistep methods with variable step size is covered in \cite{VSSLMMpaper}.

\subsection{Previous work}
Contractive linear multistep methods were studied by Sand,
who constructed a family of implicit methods with arbitrarily high order
\cite[Theorem~2.3]{Sand_1986} whose number of steps is exponential in the
order of accuracy.
Later Lenferink deduced many properties of
the optimal methods and threshold factors for low order methods \cite{Lenferink_1989,Lenferink_1991}.  
More recently, such methods have been studied by Hundsdorfer \& Ruuth 
\cite{Hundsdorfer_Ruuth_2005,Ruuth_Hundsdorfer_2005}, who consider the 
effect of special starting procedures and of requiring only boundedness
rather than strict contractivity or monotonicity.  A fast algorithm for
computing optimal methods, along with extensive results, was given in
\cite{Ketcheson_2009}.

\subsection{Scope and main results}
Linear multistep methods are closely related to polynomial interpolation
formulas.  We investigate their properties herein using the framework of
linear programming and results on polynomial interpolants.

The main results proved in the present work are:
\begin{enumerate}
    \item Existence of arbitrary order explicit LMMs with $\C>0$;
    \item A sharper upper bound on $\C$ for implicit LMMs;
    \item Behavior of the optimal value of $\C$ for $k$-step methods as $k\to \infty$.
\end{enumerate}
Along the way, we also give a new proof of the known upper bound on $\C$ for
explicit LMMs, and a new relation between $\C$ for certain implicit and explicit
classes of LMMs.

\section{Formulation as a linear programming feasibility problem}
By replacing $u_n$ with the exact solution $u(t_n)$ in \eqref{lmm} and expanding
terms in Taylor series about $t_{n}$, we obtain the following conditions for method
\eqref{lmm} to be consistent of order $p$:
\begin{subequations}
\begin{align}
  \sum_{j=1}^{k} \alpha_j & = 1 \label{primal:oc1}\\
  \sum_{j=1}^{k} \left(j^m \alpha_j - m j^{m-1} \beta_j\right) - m 0^{m-1}\beta_0& = 0 & 1 \le m \le p;\label{primal:oc2}
\end{align}
\end{subequations}
here and in the rest of the paper, we assign $0^0$ the value $1$.

A linear multistep method \eqref{lmm} has SSP coefficient at least equal to $r>0$ iff
\begin{equation*}
  \beta_j\geq 0 \text{ and } \alpha_j-r\beta_j \geq 0 \text{ for } 1 \leq j\leq k, \text{ and } \beta_0 \ge 0.
\end{equation*}
Thus the problem of whether there exists a method of order $p$ with $k$ steps and 
SSP coefficient at least $r > 0$ can be formulated for explicit methods as \cite[LP 2]{Ketcheson_2009}:
\begin{align}
  \text{Find}\label{lp:primal:exp}\\
  \beta_j \ge 0 \quad \delta_j & \ge 0 && \text{ for } 1\le j \le k\notag\\
  \text{ such that }\notag\\
  \sum_{j=1}^{k}\left( (\delta_j + r\beta_j)j^m - \beta_j m j^{m-1} \right) & = 0^m && \text{ for } 0\le m\le p;\notag
\end{align}
and for implicit methods as \cite[LP 4]{Ketcheson_2009}:
\begin{align}
  \text{Find}\label{lp:primal:imp}\\
  \beta_0 \geq 0 \quad \text{and}\quad \beta_j \ge 0\quad \delta_j &\ge 0 & &\text{ for } 1\le j \le k\notag\\
  \text{ such that }\notag\\
  \sum_{j=1}^{k} (\delta_j + r\beta_j) & = 1 && \text{ and }\notag\\
  \sum_{j=1}^{k}\left( (\delta_j + r\beta_j)j^m - \beta_j m j^{m-1} \right) - \beta_0 0^{m-1}& = 0 & &\text{ for } 1\le m\le p.\notag
\end{align}
In both cases, $\alpha_j$ can be obtained for a fixed $r$ via $\alpha_j = \delta_j + r\beta_j$
for all $1 \le j \le k$.

\section{Optimality conditions for SSP methods}
In this section we develop the basic tools used in this paper. 
Our analysis relies on Farkas' lemma and on the Duality and Complementary
slackness theorems in linear programming (LP), which we now recall; see, e.g.
\cite{schrijver1998theory}.

\begin{prop}[Farkas' lemma]
  \label{prop:farkas}
  Let $A\in\Real^{m\times n}$ be a matrix and let $b\in\Real^m$ be a vector.
  The system $Ax = b, x \geq 0$ is feasible if and only if the system
  $A^\top y \ge 0, b^\top y < 0$ is infeasible.
\end{prop}

\begin{prop}[Duality theorem in linear programming]
  \label{prop:lp-duality}
  Let $A \in \Real^{m\times n}$ be a matrix and let $b \in \Real^m, c \in \Real^n$ be vectors.
  Consider the primal-dual pair of LP problems
  \begin{subequations}\label{eq:dual-primal}
  \begin{align}
    \text{Maximize } c^\top x &\text{ subject to } Ax = b \text{ and } x\ge 0,\label{eq:primal}\\
    \text{Minimize } b^\top y &\text{ subject to } A^\top y \ge c.\label{eq:dual}
  \end{align}
  \end{subequations}
  If both problems are feasible, then the two optima are equal. 
\end{prop}

\begin{prop}[Complementary slackness theorem]
  \label{prop:compslack}
  Let $A \in \Real^{m \times n}$ be a matrix and let $b \in \Real^m, c \in \Real^n$ be vectors.
  Consider the primal-dual pair of LP problems \eqref{eq:primal} and \eqref{eq:dual}.
  Assume that both optima are finite; let $\tilde{x}$ be an optimal solution to \eqref{eq:primal}
  and let $\tilde{y}$ be an optimal solution to \eqref{eq:dual}.
  Then $\tilde{x}^\transp (b-A^\top \tilde{y}) = 0$, i.e.\ if a component of $\tilde{x}$ is positive,
  then the corresponding inequality in $A^\transp y \ge c$ is satisfied by $\tilde{y}$ with equality.
\end{prop}

Our analysis also relies on the following well-known result of Linear Algebra that can be found in many textbooks,
for instance, in \cite{spence2008elementary}.

\begin{prop}[Rouch\'e--Capelli theorem]
  \label{prop:rouchecapelli}
  Let $A \in \Real^{m \times n}$ be a matrix and let $b \in \Real^m$ be a vector.
  The system $Ax = b$ has a solution if and only if the rank of its coefficient matrix $A$
  is equal to the rank of its augmented matrix, $(A|b)$.
  If a solution exists, then the set of solutions forms an affine subspace of $\Real^n$ of dimension $(n-\rank{A})$.
\end{prop}

We can now state three simple lemmas. Since the system \eqref{dual:lp:exp:a}--\eqref{dual:lp:exp:c} below is the Farkas-dual
of \eqref{lp:primal:exp}, and the system formed by \eqref{dual:lp:exp:a}--\eqref{dual:lp:exp:c}
and \eqref{dual:lp:imp:d} is the Farkas-dual of \eqref{lp:primal:imp},
Lemma \ref{lem:farkas:exp} and Lemma \ref{lem:farkas:imp} are consequences of
Farkas's lemma (Proposition \ref{prop:farkas}) and the observation that the dual
variables can be viewed as the coefficients of a polynomial of degree at most $p$.

\begin{lem}
  \label{lem:farkas:exp}
  Let $r\in \Real$ and let $k,p$ be positive integers.
  The following statements are equivalent:
  \begin{enumerate}[(i)]
    \item
      The LP feasibility problem \eqref{lp:primal:exp} is infeasible.
    \item
      There exist $y_m\in \Real$ for $0\le m \le p$ that satisfy the system
      \begin{subequations} \label{dual:lp:exp}
      \begin{align}
        \sum_{m=0}^p j^my_m & \geq 0 && \text{ for } 1\leq j \leq k \label{dual:lp:exp:a}\\
        \sum_{m=0}^p \left(-mj^{m-1}y_m+rj^my_m\right) & \geq 0 && \text{ for } 1\leq j \leq k \label{dual:lp:exp:b}\\
        y_0 & < 0. \label{dual:lp:exp:c}
      \end{align}
      \end{subequations}
    \item
      There exists a real polynomial $q$ of degree at most $p$ that satisfies the conditions
      \begin{subequations}\label{dual:poly:exp}
      \begin{align}
        q(j) & \ge 0 && \text{ for } 1\leq j \leq k \label{dual:poly:exp:a}\\
        -q'(j)+r\cdot q(j) & \ge 0 && \text{ for } 1\leq j \leq k \label{dual:poly:exp:b}\\
        q(0) & < 0. \label{dual:poly:exp:c}
      \end{align}
      \end{subequations}
  \end{enumerate}
\end{lem}

\begin{lem}
  \label{lem:farkas:imp}
  Let $r\in \Real$ and let $k,p$ be positive integers.
  The following statements are equivalent:
  \begin{enumerate}[(i)]
    \item
      The LP feasibility problem \eqref{lp:primal:imp} is infeasible.
    \item
      There exist $y_m\in \Real$ for $0\le m \le p$ that satisfy the system
      defined by \eqref{dual:lp:exp:a}--\eqref{dual:lp:exp:c} and
      \begin{equation} \label{dual:lp:imp:d}
        -y_1 \geq 0. \tag{\ref*{dual:lp:exp}d}
      \end{equation}
    \item
      There exists a real polynomial $q$ of degree at most $p$ that
      satisfies the system defined by \eqref{dual:poly:exp:a}--\eqref{dual:poly:exp:c} and
      \begin{equation} \label{dual:poly:imp:d}
        -q'(0) \ge 0. \tag{\ref*{dual:poly:exp}d}
      \end{equation}
  \end{enumerate}
\end{lem}

Recall that due to the Fundamental theorem of algebra,
a real polynomial can be factored in a unique way as $q(x) = c\cdot \prod_{m=1}^{p_0}(x-\lambda_m)$,
where $c \in \Real \setminus \{0\}$ is the leading coefficient of $q$ and
$\lambda_1, \lambda_2, \ldots, \lambda_{p_0} \in \Comp$ are the roots of $q$.
Moreover, in view of the Complex conjugate root theorem, the non-real roots of $q$ occur in complex conjugate pairs.
The proof of Lemma \ref{lem:dual:roots} is also based on the formula
$\frac{q'(x)}{q(x)} = \sum_{m=1}^{p_0}\frac{1}{(x-\lambda_m)}$
for the logarithmic derivative of the polynomial $q(x) := c\cdot \prod_{m=1}^{p_0}(x-\lambda_m)$.

\begin{lem} \label{lem:dual:roots}
  Let $r\in \Real$,
  and let $q$ be a non-zero real polynomial of degree $p_0 > 0$.
  Let $c \in \Real \setminus \{0\}$ be the leading coefficient of $q$ and
  $\lambda_1, \lambda_2, \ldots, \lambda_{p_0} \in \Comp$ be the roots of $q$.
  For any real $\mu$, let $s(\mu) := \left|\{m : 1\le m \le p_0, \text{ and } \lambda_m\in \Real, \text{ and }\lambda_m \ge \mu \}\right|$
  denote the number of real roots of $q$ that are greater than or equal to $\mu$.
  Then the following statements hold.
  \begin{subequations}
  \begin{enumerate}[(i)]
    \item \label{roots:s1}
      The polynomial $q$ satisfies the inequality \eqref{dual:poly:exp:a} for some integer $j$ if and only if either of the following is true:
      \begin{itemize}
        \item $j$ is a root of $q$;
        \item $j$ is not a root of $q$, and
          \begin{equation} \label{dual:roots:exp:a}
            c \cdot (-1)^{s(j)} > 0.
          \end{equation}
      \end{itemize}
    \item \label{roots:s2}
      Suppose that $q$ satisfies the inequality \eqref{dual:poly:exp:a} for some integer $j$.
      Then $q$ satisfies the inequality \eqref{dual:poly:exp:b} for the same value of $j$ if
      and only if either of the following is true:
      \begin{itemize}
        \item $j$ is a multiple root of $q$;
        \item $j$ is a simple root of $q$, and \eqref{dual:roots:exp:a} holds for $j$;
        \item $j$ is not a root of $q$, and
          \begin{equation} \label{dual:roots:exp:b}
            \sum_{m=1}^{p_0}\frac{1}{j-\lambda_m} \le r.
          \end{equation}
      \end{itemize}
    \item \label{roots:s3}
      The polynomial $q$ satisfies \eqref{dual:poly:exp:c} if and only if $0$
      is not a root of $q$ and
      \begin{equation} \label{dual:roots:exp:c}
        c \cdot (-1)^{s(0)} < 0.
      \end{equation}
    \item \label{roots:s4}
      Suppose that $q$ satisfies \eqref{dual:poly:exp:c}.
      Then $q$ satisfies \eqref{dual:poly:imp:d} if and only if  
      \begin{equation} \label{dual:roots:imp:d}
        \sum_{m=1}^{p_0}\frac{1}{\lambda_m} \le 0.
      \end{equation}
  \end{enumerate}
\end{subequations}
\end{lem}

For convenience, we introduce the following notation.
Let $\Cexp(k,p)$ denote the ``optimal'' SSP coefficient for explicit LM methods with $k$ steps and
order of accuracy $p$, i.e.\
\begin{equation*}
  \Cexp(k,p) = \sup_{\substack{\text{$k$-step explicit}\\ \text{methods of order $p$}}}\C,
\end{equation*}
and let $\Cimp(k,p)$ denote the ``optimal'' SSP coefficient for implicit LM methods with $k$ steps and
order of accuracy $p$, i.e.\
\begin{equation*}
  \Cimp(k,p) = \sup_{\substack{\text{$k$-step implicit}\\ \text{methods of order $p$}}}\C.
\end{equation*}

Our purpose is to study the behavior of $\Cexp(k,p)$ and $\Cimp(k,p)$.
It turns out to be convenient to study polynomials that satisfy the following
conditions, in place of \eqref{dual:poly:exp:a}--\eqref{dual:poly:imp:d}:
      \begin{subequations} \label{dual:poly2:exp}
      \begin{align}
        q(j) &\ge 0 && \text{ for } 1\leq j \leq k \label{dual:poly2:exp:a}\\
        -q'(j)+r\cdot q(j) & \ge 0 && \text{ for } 1\leq j \leq k \label{dual:poly2:exp:b}\\
        q(0) &= 0 \label{dual:poly2:exp:c} \\
        q'(0) &= 0. \label{dual:poly2:exp:d}
      \end{align}
      \end{subequations}
Before we proceed to our main lemmas, Lemma \ref{lem:dualcond:exp} and Lemma \ref{lem:dualcond:imp},
we need to prove two auxiliary results relating these sets of conditions.

\begin{lem} \label{lem:aux:exp}
  Let $r \in \Real$ and let $k, p$ be positive integers.
  Suppose that $q$ is a real polynomial of degree $0 < p_0 \le p$ that satisfies
  the system \eqref{dual:poly:exp:a}--\eqref{dual:poly:exp:c}.
  Then there exists a non-zero real polynomial
  \begin{equation*}
    \tilde{q}(x) := \tilde{c} \cdot \prod_{m=1}^p (x-\tilde{\lambda}_m)
  \end{equation*}
  of degree $p$ that satisfies the system \eqref{dual:poly2:exp:a}--\eqref{dual:poly2:exp:c}, has
  a leading coefficient $\tilde{c} = \pm1$, and whose roots satisfy
  $\tilde{\lambda}_m \in \mathcal{H} := \left\{z \in \Comp : \vert\Im(z)\vert \le \frac{k-1}{2}, \text{ and } 0\le \Re(z)\le k \right\}$
  for all $1 \le m \le p$.
\end{lem}
\begin{proof}
  Let $r \in \Real$, and let $k, p$ be positive integers.
  Suppose that $q$ is a real polynomial of degree $p_0 \le p$ that satisfies the system
  \eqref{dual:poly:exp:a}--\eqref{dual:poly:exp:c}.
  Let us consider the polynomial $\hat{q}(x) := q(x) - q(0)$.
  This is a non-zero real polynomial of degree $p_0$ that satisfies
  the system \eqref{dual:poly2:exp:a}--\eqref{dual:poly2:exp:c}; furthermore, it
  fulfills the condition $\hat{q}(j) > 0$ for all $1 \le j \le k$.
  Let $\hat{q}(x) := \hat{c} \cdot \prod_{m=1}^{p_0} (x-\hat{\lambda}_m)$ be the
  factored form of $\hat{q}$, let $M := |\{\hat{\lambda}_m : 1 \le m \le p_0, \text{ and } \hat{\lambda}_m \in \Real, \text{ and } \hat{\lambda}_m < 0\}|$
  be the number of negative real roots of $\hat{q}$, and let us take
  \begin{align*}
    \tilde{c} & := (-1)^{M+p-p_0}\cdot \frac{\hat{c}}{|\hat{c}|} \\
    \tilde{\lambda}_m & :=
      \begin{cases}
        \hat{\lambda}_m & \text{ if } 1 \le m \le p_0 \text{ and } \hat{\lambda}_m \in \mathcal{H} \\
        k & \text{ if } 1 \le m \le p_0 \text{ and } \hat{\lambda}_m \not\in \mathcal{H}, \text{ or if } p_0 < m \le p
      \end{cases}\\
      \tilde{q}(x) & := \tilde{c}\cdot \prod_{m=1}^{p} (x-\tilde{\lambda}_m).
  \end{align*}
  It is clear that $\tilde{q}$ satisfies condition \eqref{dual:poly2:exp:c}, 
  and that $\tilde{c} = \pm1$.

  Let $\hat{s}$ and $\tilde{s}$ be defined as in Lemma \ref{lem:dual:roots} with
  respect to $\hat{q}$ and $\tilde{q}$, respectively.
  Because of statements \eqref{roots:s1} and \eqref{roots:s2} of Lemma \ref{lem:dual:roots},
  we have $\hat{c}\cdot (-1)^{\hat{s}(j)} > 0$
  and $\sum_{m=1}^{p_0}(j-\hat{\lambda}_m)^{-1} \le r$ for all $1 \le j \le k$.
  Since the non-real roots of $\hat{q}$ come in complex conjugate pairs,
  $\tilde{s}(j) \equiv \hat{s}(j) + M + p-p_0 \mod 2$
  for all $1 \le j \le k$. Hence
  $\tilde{c} \cdot (-1)^{\tilde{s}(j)} = {\hat{c}}/{|\hat{c}|} \cdot (-1)^{\hat{s}(j)} > 0$
  for all $1 \le j \le k$, and thus
  statement \eqref{roots:s1} of Lemma \ref{lem:dual:roots} implies that  $\tilde{q}$
  satisfies condition \eqref{dual:poly2:exp:a}.

  Now we prove that $\tilde{q}$ satisfies condition \eqref{dual:poly2:exp:b}.
  To achieve this, we show that the contribution of the roots of $\tilde{q}$
  to the left hand side of \eqref{dual:roots:exp:b} is less than or equal to
  that of the roots of $\hat{q}$.

  Let us consider first a $\mu \in \Comp$ such that $|\Im(\mu)| > \frac{k-1}{2}$.
  Then
  \begin{equation*}
    \frac{1}{j-\mu}+\frac{1}{j-\conj{\mu}} = \frac{2(j-\Re(\mu))}{(j-\Re(\mu))^2+(\Im(\mu))^2}
    \geq -\frac{1}{|\Im(\mu)|} > -\frac{2}{k-1} \ge \frac{2}{j-k}
  \end{equation*}
  for all $1 \le j \le k-1$.

  Let us consider now a $\mu \in \Comp$ such that either $\Re(\mu) > k$ or $\Re(\mu) < 0$.
  Then
  \begin{equation*}
    \frac{1}{j-\mu}+\frac{1}{j-\conj{\mu}} = \frac{2(j-\Re(\mu))}{(j-\Re(\mu))^2+(\Im(\mu))^2}
    \geq -\frac{2}{\Re(\mu) - j} > \frac{2}{j-k}
  \end{equation*}
  for all $1 \le j \le k-1$.

  It follows that 
  \begin{equation*}
    \sum_{m=1}^{p}\frac{1}{j-\tilde{\lambda}_m} \le \sum_{m=1}^{p_0}\frac{1}{j-\tilde{\lambda}_m} \le \sum_{m=1}^{p_0}\frac{1}{j-\hat{\lambda}_m} \le r
  \end{equation*}
  for all $1 \le j \le k-1$.
  If $k$ is not a root of $\tilde{q}$ then the same holds for $j = k$;
  otherwise $k$ is a root of $\tilde{q}$ and $\tilde{c}\cdot (-1)^{\tilde{s}(k)} > 0$.
  In both cases, statement \eqref{roots:s2} of Lemma \ref{lem:dual:roots} implies that $\tilde{q}$
  satisfies condition \eqref{dual:poly2:exp:b}, which completes the proof.
\end{proof}

\begin{lem} \label{lem:aux:imp}
  Let $r \in \Real$ and let $k, p$ be positive integers.
  Suppose that $q$ is a real polynomial of degree $0 < p_0 \le p$ that satisfies
  the system formed by \eqref{dual:poly:exp:a}--\eqref{dual:poly:exp:c} and \eqref{dual:poly:imp:d}.
  Then there exists a non-zero real polynomial
  \begin{equation*}
    \tilde{q}(x) := \tilde{c} \cdot \prod_{m=1}^p (x-\tilde{\lambda}_m)
  \end{equation*}
  of degree $p$ that satisfies the system \eqref{dual:poly2:exp:a}--\eqref{dual:poly2:exp:d}, has
  a leading coefficient $\tilde{c} = \pm 1$, and whose roots satisfy
  $\tilde{\lambda}_m \in \mathcal{H} := \left\{z \in \Comp : \vert\Im(z)\vert \le \frac{k-1}{2}, \text{ and } 0\le \Re(z)\le k \right\}$
  for all $1 \le m \le p$.
\end{lem}
\begin{proof}
  Let $r \in \Real$, and let $k, p$ be positive integers.
  Suppose that $q$ is a real polynomial of degree $p_0 \le p$ that satisfies the system
  formed by \eqref{dual:poly:exp:a}--\eqref{dual:poly:exp:c} and \eqref{dual:poly:imp:d}.
  Let us consider the polynomial $\hat{q}(x) := q(x)-q(0)$.
  This is a non-zero real polynomial of degree $p_0$ that satisfies
  the system formed by \eqref{dual:poly2:exp:a}--\eqref{dual:poly2:exp:c} and \eqref{dual:poly:imp:d},
  furthermore it fulfills the condition $\hat{q}(j) > 0$ for all $1 \le j \le k$.

  Let $\hat{q}(x) := \hat{c} \cdot \prod_{m=1}^{p_0} (x-\hat{\lambda}_m)$ be the
  factored form of $\hat{q}$, and let $\hat{s}$ be as in Lemma \ref{lem:dual:roots} with respect to $\hat{q}$.
  Statement \eqref{roots:s1} of Lemma \ref{lem:dual:roots} implies that $\hat{c}\cdot (-1)^{\hat{s}(1)} > 0$
  and statement \eqref{roots:s2} of Lemma \ref{lem:dual:roots} implies that either $\hat{c}\cdot (-1)^{\hat{s}(0)} > 0$
  or $0$ is root of $\hat{q}$ of multiplicity at least two.
  In both cases there are two distinct indices $1 \le m_0 \ne m_1 \le p_0$, such that $\hat{\lambda}_{m_0} = 0$,
  and $\hat{\lambda}_{m_1} \in \Real$, and $0 \le \hat{\lambda}_{m_1} < 1$.
  Let $M$ be the number of negative real roots of $\hat{q}$, as in the proof of Lemma \ref{lem:aux:exp},
  and let us take
  \begin{align*}
    \tilde{c} & := (-1)^{M+p-p_0}\cdot \frac{\hat{c}}{|\hat{c}|} \\
    \tilde{\lambda}_m & :=
      \begin{cases}
        \hat{\lambda}_m & \text{ if } 1 \le m \le p_0 \text{ and } \hat{\lambda}_m \in \mathcal{H} \text{ and } m \ne m_1\\
        0 & \text{ if } m = m_1 \\
        k & \text{ if } 1 \le m \le p_0 \text{ and } \hat{\lambda}_m \not\in \mathcal{H}, \text{ or if } p_0 < m \le p
      \end{cases}\\
      \tilde{q}(x) & := \tilde{c}\cdot \prod_{m=1}^{p} (x-\tilde{\lambda}_m).
  \end{align*}

  We can proceed in the same way as in Lemma \ref{lem:aux:exp} to show that this
  $\tilde{q}$ satisfies the claims of the lemma.
\end{proof}

\begin{lem} \label{lem:dualcond:exp}
  Let $r \ge 0$ and let $k, p$ be positive integers.
  Then the following statements are equivalent.
  \begin{enumerate}[(i)]
    \item \label{dualcond:exp:s1}
      $\Cexp(k,p) \le r$.
    \item \label{dualcond:exp:s2}
      There exists a non-zero real polynomial $q$ of degree at most $p$ that satisfies
      the system \eqref{dual:poly2:exp:a}--\eqref{dual:poly2:exp:c}.
  \end{enumerate}
\end{lem}
\begin{proof}
  We can check that the polynomial $q(x) :=  x-\frac12$ of degree one satisfies the system
  \eqref{dual:poly:exp:a}--\eqref{dual:poly:exp:c} with $r = 2$ for any value of $k$,
  and thus Lemma \ref{lem:farkas:exp} implies that $\Cexp(k, p) \le 2$, and hence it is finite,
  for all positive integers $k$ and $p$.
  Now we are ready to show the equivalence of statements \eqref{dualcond:exp:s1} and \eqref{dualcond:exp:s2}.

  \textbf{Proof that statement \eqref{dualcond:exp:s2} implies statement \eqref{dualcond:exp:s1}.}
  Suppose that $q:= c \cdot \prod_{m=1}^{p_0} (x - \lambda_m)$ is a non-zero real polynomial of degree $0 < p_0 \le p$ that
  satisfies the system \eqref{dual:poly2:exp:a}--\eqref{dual:poly2:exp:c} with
  $r = \rho_0 \ge 0$ for some positive integers $k$ and $p$, but $\Cexp(k, p) = \rho_1 > \rho_0$.
  Since $0$ is a root of $q$, we can assume without the loss of generality that
  $\lambda_1, \ldots, \lambda_{p_1} \ne 0$ and $\lambda_{p_1+1}, \ldots, \lambda_{p_0} = 0$
  for some integer $0 \le p_1 < p_0$.

  Let us choose $0 < \epsilon < 1$ so that $\frac{\epsilon}{1-\epsilon} \le \frac{\rho_1-\rho_0}{2}$,
  and let us take the polynomial
  \begin{equation*}
    \tilde{q}(x) :=
    \begin{cases}
      c \cdot \prod_{m=1}^{p_1} (x - \lambda_m) & \text{ if } c \cdot \prod_{m=1}^{p_1} ( - \lambda_m) < 0, \\
      c \cdot (x-\epsilon)\prod_{m=1}^{p_1} (x - \lambda_m) & \text{ if } c \cdot \prod_{m=1}^{p_1} ( - \lambda_m) > 0.
    \end{cases}
  \end{equation*}
  Let $s$ and $\tilde{s}$ be as in Lemma \ref{lem:dual:roots}
  with respect to $q$ and $\tilde{q}$, respectively.
  By construction, both the leading coefficients and the set of real roots greater than $\epsilon$ of
  $q$ and $\tilde{q}$ coincide, and thus $\tilde{s}(j) = s(j)$ for all $1 \le j \le k$.

  We can check that $\tilde{q}$ satisfies conditions \eqref{dual:poly:exp:a} and
  \eqref{dual:poly:exp:c}.
  Moreover, in view of Lemma \ref{lem:dual:roots}, if $j$ is not a root of $q$ for some $1 \le j \le k$
  then \eqref{dual:roots:exp:b} holds for the roots of $q$ with $r = \rho_0$, and thus
  \begin{align*}
    \rho_0 & \ge \sum_{m=1}^{p_0} \frac1{j-\lambda_m} = \sum_{m=1}^{p_1} \frac{1}{j-\lambda_m}
    + (p_0-p_1)\frac{1}{j} \ge \sum_{m=1}^{p_1} \frac{1}{j-\lambda_m} + \frac{1}{j} \\
    & = \sum_{m=1}^{p_1} \frac1{j-\lambda_m} + \frac{1}{j-\epsilon}
    - \frac{\epsilon}{j(j-\epsilon)} \ge 
    \sum_{m=1}^{p_1} \frac1{j-\lambda_m} + \frac{1}{j-\epsilon} - \frac{\rho_1-\rho_0}{2} \\
    \frac{\rho_0+\rho_1}{2} & \ge \sum_{m=1}^{p_1} \frac1{j-\lambda_m} + \frac{1}{j-\epsilon}.
  \end{align*}

  Hence, in view of statement \eqref{roots:s2} of Lemma \ref{lem:dual:roots}, $\tilde{q}$ satisfies condition
  \eqref{dual:poly:exp:b} with $r = \frac{\rho_0 + \rho_1}{2}$.
  Now Lemma \ref{lem:farkas:exp} implies that $\C < \frac{\rho_0+\rho_1}{2}$ for any $k$-step
  explicit LM methods with order of accuracy $p$, which leads to the contradiction
  $\Cexp(k, p) \le \frac{\rho_0+\rho_1}{2} < \rho_1$.
  This completes the first part of the proof.

  \textbf{Proof that statement \eqref{dualcond:exp:s1} implies statement \eqref{dualcond:exp:s2}.}
  Let $k, p$ be positive integers.
  We can check that if a non-zero real polynomial $q$ satisfies the system
  \eqref{dual:poly2:exp:a}--\eqref{dual:poly2:exp:c} with $r = \rho$
  then it satisfies the same system with any $r > \rho$. Hence it is enough to
  prove that there is a non-zero real polynomial $q$ of degree at most $p$ that satisfies the system
  \eqref{dual:poly2:exp:a}--\eqref{dual:poly2:exp:c} with $r = \Cexp(k,p)$.

  Let $\mathcal{H}$ be as in Lemma \ref{lem:aux:exp}.
  By the definition of $\Cexp(k, p)$ and due to Lemma \ref{lem:farkas:exp}, there
  is a real sequence $(r_n)_{n = 1}^\infty$ and a corresponding sequence of
  real polynomials $(q_n)_{n = 1}^\infty$ of degree at most $p$, such that
  $\lim_{n\to \infty} r_n = \Cexp(k, p)$, and for all $n$, the polynomial $q_n$
  satisfies the system \eqref{dual:poly:exp:a}--\eqref{dual:poly:exp:c} with $r = r_n$.
  Due to Lemma \ref{lem:aux:exp}, for any positive integer $n$, there exists a
  non-zero real polynomial
  $\tilde{q}_n(x) := \tilde{c}_n \cdot \prod_{m=1}^p (x - \tilde{\lambda}_{n, m})$
  of degree $p$, such that $\tilde{q}_n$ satisfies the system
  \eqref{dual:poly2:exp:a}--\eqref{dual:poly2:exp:c} with $r=r_n$, moreover
  $\tilde{c}_n = \pm 1$ and $\tilde{\lambda}_{n, m} \in \mathcal{H}$ for all $1 \le m \le p$.

  Since $(\tilde{\lambda}_{n, 1}, \tilde{\lambda}_{n, 2}, \ldots, \tilde{\lambda}_{n, p}) \in \mathcal{H}^p$
  for all $n$, and since $\mathcal{H}^p \subset \Comp^p$ is closed and bounded,
  the Bolzano--Weiertrass theorem implies the existence of a strictly increasing 
  sequence of indices $(n_\ell)_{\ell=1}^\infty$, such that for all $1 \le m \le p$, the limit
  $\tilde{\lambda}_{m} := \lim_{\ell \to \infty} \tilde{\lambda}_{n_\ell, m}$
  exists, $\tilde{\lambda}_m \in \mathcal{H}$ and non-real $\tilde{\lambda}$ occur in
  complex conjugate pairs, moreover
  $\tilde{c}_{n_\ell} = \tilde{c}$ for all $\ell$.
  A continuity argument yields that
  $\tilde{q}(x) := \tilde{c}\cdot \prod_{m=1}^p (x - \tilde{\lambda}_m)$ is a real polynomial of degree $p$ that
  satisfies the system \eqref{dual:poly2:exp:a}--\eqref{dual:poly2:exp:c}
  with $r = \Cexp(k, p)$, which completes the proof.
\end{proof}

\begin{lem} \label{lem:dualcond:imp}
  Let $r \ge 0$ and let $k$, $p$ be positive integers.
  Then the following statements are equivalent.
  \begin{enumerate}[(i)]
    \item \label{dualcond:imp:s1}
      $\Cimp(k,p) \le r$.
    \item \label{dualcond:imp:s2}
      There exists a non-zero real polynomial $q$ of degree at most $p$ that satisfies
      the system \eqref{dual:poly2:exp:a}--\eqref{dual:poly2:exp:d}.
  \end{enumerate}
\end{lem}
\begin{proof}
  The proof is essentially the same as for Lemma \ref{lem:dualcond:exp}.
\end{proof}

The results of this paper are based on analysing the polynomials $q$ appearing
in Lemma \ref{lem:dualcond:exp} and in Lemma \ref{lem:dualcond:imp},
and for this reason, we establish some basic properties of these polynomials in
Lemma \ref{lem:dualpropx:exp} and in Lemma \ref{lem:dualpropx:imp}.

\begin{lem} \label{lem:dualpropx:exp}
  Let $k, p$ be positive integers and let $r \ge \Cexp(k, p)$ be a real number.
  Then there exists a non-zero real polynomial $q$ that solves the system
  \eqref{dual:poly2:exp:a}--\eqref{dual:poly2:exp:c} and has the following properties.
  \begin{enumerate}
      \item The degree of $q$ is exactly $p$. \label{prop1}
      \item The real parts of all roots of $q$ except for $0$ lie in the interval $[1,k]$. \label{prop2}
      \item All real roots of $q$ are integers. \label{prop3}
      \item The multiplicity of the root $0$ is $1$. For $1\leq j \leq k-1$,
          if $j$ is a root of $q$ then its multiplicity is $2$. \label{prop5}
      \item For even $p$, $k$ is a root of $q$ of multiplicity $1$. For odd $p$, if $k$ is a root of $q$ then its multiplicity is $2$. \label{prop4}
      \item The polynomial $q$ is non-negative on the interval $[0, k]$. \label{prop6}
  \end{enumerate}
  Moreover, if $\Cexp(k, p) > 0$ and $r = \Cexp(k, p)$ then every non-zero real
  polynomial $q$ of degree at most $p$ that solves the system
  \eqref{dual:poly2:exp:a}--\eqref{dual:poly2:exp:c} has these properties.
\end{lem}
\begin{proof}
  \textbf{Case $r = \Cexp(k, p) > 0$.}
  Suppose that positive integers $k, p$ are such that $\Cexp(k, p) > 0$.
  We know that $\Cexp(k, p) < \infty$, see the proof of Lemma \ref{lem:dualcond:exp}.
  Let $q(x) := c \cdot \prod_{m=1}^{p_0} (x-\lambda_m)$ be a non-zero real polynomial of degree $p_0 \le p$ that satisfies
  the system \eqref{dual:poly2:exp:a}--\eqref{dual:poly2:exp:c} with $r = \rho_0 := \Cexp(k, p)$;
  the existence of such a polynomial is guaranteed by Lemma \ref{lem:dualcond:exp}.
  Due to condition \eqref{dual:poly2:exp:c}, $\lambda_{m_0} = 0$ for some $1 \le m_0 \le p_0$.
  Now we show that $q$ has all the required properties.

  \textbf{Proofs for properties \ref{prop1} and \ref{prop2}.}
  Suppose that either property \ref{prop1} or \ref{prop2} does not hold for $q$.
  Let $M := |\{m: 1 \le m \le p_0, \text{ and } m \ne m_0, \text{ and } \lambda_m \not\in [1, k] \}|$ be the number of roots of $q$ with real parts outside the interval $[1,k]$, not counting $\lambda_{m_0}$.
  Let us take
  \begin{align*}
    \tilde{c} & := (-1)^{M+p-p_0}\cdot c \\
    \tilde{\lambda}_m & :=
      \begin{cases}
        \lambda_m & \text{ if } 1 \le m \le p_0 \text{ and } \Re(\lambda_m) \in [1, k], \text{ or if } m = m_0 \\
        k & \text{ if } 1 \le m \le p_0 \text{ and } m \ne m_0 \text{ and } \Re(\lambda_m) \not\in [1, k], \text{ or if } p_0 < m \le p
      \end{cases}\\
      \tilde{q}(x) & := \tilde{c}\cdot \prod_{m=1}^{p} (x-\tilde{\lambda}_m).
  \end{align*}
  By following the argument in the proof of Lemma \ref{lem:aux:exp},
  we can show that $\tilde{q}$ satisfies the system \eqref{dual:poly2:exp:a}--\eqref{dual:poly2:exp:c}
  with $r = \rho_1$ for some $0 \le \rho_1 < \rho_0$.
  Now Lemma \ref{lem:dualcond:exp} yields $\Cexp(k, p) \le \rho_1 < \rho_0$, which
  contradicts our choice of $q$.

  \textbf{Proof for property \ref{prop3}.}
  Suppose that property \ref{prop3} does not hold for $q$.
  Let us take
  \begin{align*}
    \tilde{\lambda}_m & :=
    \begin{cases}
      \lambda_m & \text{ if } 1 \le m\le p_0 \text{ and } \lambda_m \not\in \Real \\
      \lfloor \lambda_m \rfloor & \text{ if } 1 \le m\le p_0 \text{ and } \lambda_m \in \Real
    \end{cases} \\
    \tilde{q}(x) & := c \cdot \prod_{m=1}^{p_0} (x-\tilde{\lambda}_m).
  \end{align*}
  Let $s$ and $\tilde{s}$ be as in Lemma \ref{lem:dual:roots} with respect to
  $q$ and $\tilde{q}$, respectively.
  We can check that $\tilde{s}(j) = s(j)$ for all $0 \le j \le k$.
  Moreover, if $j$ is a root of $q$ for some integer $1 \le j \le k$ then
  $j$ is a root of $\tilde{q}$ too, with at least the same multiplicity.
  If, however, $j$ is not a root of $\tilde{q}$ for some integer $0 \le j \le k$
  then, assuming that at least one root has moved, we have
  $\sum_{m = 0}^{p_0} (j-\lambda_m)^{-1} > \sum_{m = 0}^{p_0} (j-\tilde{\lambda}_m)^{-1}$.
  In view of Lemma \ref{lem:dual:roots}, the polynomial $\tilde{q}$ satisfies
  the system \eqref{dual:poly2:exp:a}--\eqref{dual:poly2:exp:c} with $r = \rho_1$
  for some $0 \le \rho_1 < \rho_0$, which again leads to a contradiction.

  \textbf{Proof for property \ref{prop5}.}
  We already know that properties \ref{prop1}--\ref{prop3} hold for $q$.
  Suppose that $\lambda_{m_1} = \lambda_{m_2} = \lambda_{m_3} = j$ for some
  integer $1 \le j \le k$ and indices $1 \le m_1 < m_2 < m_3 \le p$.
  Then we can take
  \begin{align*}
    \tilde{\lambda}_m & :=
    \begin{cases}
      \lambda_m & \text{ if } m \ne m_3 \\
      \lambda_m-1 & \text{ if } m = m_3
    \end{cases} \\
    \tilde{q}(x) & := c \cdot \prod_{m=1}^{p} (x-\tilde{\lambda}_m),
  \end{align*}
  and, as before, we can check that the polynomial $\tilde{q}$ satisfies
  the system \eqref{dual:poly2:exp:a}--\eqref{dual:poly2:exp:c} with $r = \rho_1$
  for some $0 \le \rho_1 < \rho_0$, which leads to a contradiction.
  Hence if $j$ is a root of $q$ for some integer $1 \le j \le k$ then its multiplicity
  is at most $2$.

  Suppose now that $j_0$ is a simple root of $q$ for some integer $1 \le j_0 < k$,
  and suppose further that $j$ is not a double root of $q$ for all integers $j_0 < j \le k$.
  Let $s$ be as in Lemma \ref{lem:dual:roots}, and
  let $j_0 < j_1 \le k$ be the smallest integer that is not a double root of $q$.
  Then $c\cdot (-1)^{s(j_0)} = - c \cdot (-1)^{s(j_1)}$,
  and in view of statements \eqref{roots:s1} and \eqref{roots:s2} of Lemma \ref{lem:dual:roots},
  $q$ does not satisfy the inequalities \eqref{dual:poly2:exp:a} and \eqref{dual:poly2:exp:b}
  for both $j=j_0$ and $j=j_1$.
  Therefore if $j_0$ is a simple root of $q$ for some integer $1 \le j_0 < k$
  then $j$ is a double root of $q$ for all $j_0 < j \le k$,
  and thus $\lambda_{m_1} = k$ for some index $1 \le m_1 \le p$. Now we can take
  \begin{align*}
    \tilde{\lambda}_m & :=
    \begin{cases}
      \lambda_m & \text{ if } m \ne m_1 \\
      j_0 & \text{ if } m = m_1
    \end{cases} \\
    \tilde{q}(x) & := c \cdot \prod_{m=1}^{p} (x-\tilde{\lambda}_m),
  \end{align*}
  and check that $\tilde{q}$ satisfies the system
  \eqref{dual:poly2:exp:a}--\eqref{dual:poly2:exp:c} with $r = \rho_1$
  for some $0 \le \rho_1 < \rho_0$, which again leads to a contradiction.

  \textbf{Proofs for properties \ref{prop4} and \ref{prop6}}
  Property \ref{prop4} is a consequence of properties \ref{prop1}-\ref{prop3} and \ref{prop5}, and the fact that non-real
  roots come in complex conjugate pairs.
  Property \ref{prop6} is a consequence of properties \ref{prop3} and \ref{prop5}.

  \textbf{Case $r > \Cexp(k, p) > 0$.}
  We can see that if $\Cexp(k, p)> 0$ and $q$ satisfies the system
  \eqref{dual:poly2:exp:a}--\eqref{dual:poly2:exp:c} with $r=\Cexp(k, p)$
  then $q$ satisfies the same system with any $r > \Cexp(k, p)$, and thus
  the claim of the lemma holds true in this case.

  \textbf{Case $\Cexp(k, p) = 0$ and $p < 2k+1$.}
  In this case, we can construct a non-zero real polynomial $\tilde{q}$ of degree $p$
  that satisfies the system \eqref{dual:poly2:exp:a}--\eqref{dual:poly2:exp:c}
  with $r = 0$ and has properties \ref{prop1}--\ref{prop6}, by applying the transformations presented
  in the proof of the first case to an arbitrary non-zero real polynomial $q$ of degree $p_0 \le p$
  that satisfies the same system with $r=0$.

  \textbf{Case $\Cexp(k, p) = 0$ and $p \ge 2k+1$.}
  For odd $p$ we define
  $q(x) := x(x-k+i)^{\frac{p-2k-1}{2}}(x-k-i)^{\frac{p-2k-1}{2}}\cdot \prod_{j=1}^k (x-j)^2$,
  and for even $p$ we define
  $q(x) := x(x-k+i)^{\frac{p-2k}{2}}(x-k-i)^{\frac{p-2k}{2}} (k-x)\cdot \prod_{j=1}^{k-1} (x-j)^2$.
  We can check that these polynomials satisfy all the required conditions.
\end{proof}

\begin{lem} \label{lem:dualpropx:imp}
  Let $k$ and $p \ge 2$ be positive integers and let $r \ge \Cimp(k, p)$ be a real number.
  Then there exists a non-zero real polynomial $q$ that solves the system
  \eqref{dual:poly2:exp:a}--\eqref{dual:poly2:exp:d}, and has the following properties.
  \begin{enumerate}
    \item The degree of $q$ is exactly $p$.
    \item The real parts of all roots of $q$ except for $0$ lie in the interval $[1,k]$.
    \item All real roots of $q$ are integers.
    \item The multiplicity of the root $0$ is $2$. For $1\leq j \leq k-1$,
      if $j$ is a root of $q$ then its multiplicity is $2$.
    \item For odd $p$, $k$ is a root of $q$ of multiplicity $1$. For even $p$, if $k$ is a root of $q$ then its multiplicity is $2$.
    \item The polynomial $q$ is non-negative on the interval $[0, k]$.
  \end{enumerate}
\end{lem}
  Moreover, if $\Cimp(k, p) > 0$ and $r = \Cimp(k, p)$ then every non-zero real
  polynomial $q$ of degree at most $p$ that solves the system
  \eqref{dual:poly2:exp:a}--\eqref{dual:poly2:exp:d} has these properties. 
\begin{proof}
  The proof is essentially the same as for Lemma~\ref{lem:dualpropx:exp}.
\end{proof}

The following two lemmas provide a way to analyze both the coefficient structure
and the uniqueness of ``optimal'' SSP LM methods with a given number of steps
and order of consistency.

\begin{lem}
  \label{lem:dualoptimal:exp}
  Let positive integers $k,p$  be such that $\Cexp(k,p) > 0$.
  Assume that $(\delta_j)_{j=1}^{k}$ and $(\beta_j)_{j=1}^{k}$ are the coefficients of an
  explicit LM method of order $p$ with $k$ steps and with SSP coefficient $\Cexp(k, p)$,
  and let $q$ be any non-zero polynomial of degree at most $p$ that satisfies the system
  \eqref{dual:poly2:exp:a}--\eqref{dual:poly2:exp:c} with $r=\Cexp(k,p)$.
  Let us define the sets
  $\mathcal{I} := \{j : 1 \le j \le k, \text{ and } q(j) = 0\}$ and
  $\mathcal{J} := \{j : 1 \le j \le k, \text{ and } -q'(j)+\Cexp(k,p) q(j) = 0\}$,
  and let us introduce the vectors $a_j := (j^m)_{m=0}^{p}$ for $0 \le j \le k$,
  and $b_j := (-m j^{m-1}+\Cexp(k,p) j^m)_{m=0}^{p}$ for $1 \le j \le k$.
  Then the following statements hold.
  \begin{enumerate}[(i)]
    \item \label{dualoptimal:exp:s1}
      If $\delta_j \ne 0$ for some $1\le j\le k$, then
        $q(j) = 0$.
    \item \label{dualoptimal:exp:s2}
      If $\beta_j \ne 0$ for some $1\le j\le k$, then
      $-q'(j)+\Cexp(k,p) q(j) = 0$.
    \item \label{dualoptimal:exp:s3}
      If the vectors $a_j$ for $j \in \mathcal{I}$ and $b_j$ for $j \in \mathcal{J}$
      are linearly independent then there exists no other explicit LM method of
      order $p$ with $k$ steps and SSP coefficient $\Cexp(k,p)$.
    \item \label{dualoptimal:exp:s4}
      The polynomial $q$ can be chosen so that the vectors $a_j$ for $j \in \mathcal{I}$
      and $b_j$ for $j \in \mathcal{J}$ span a $p$ dimensional subspace of $\Real^{p+1}$.
  \end{enumerate}
\end{lem}
\begin{proof}
  \textbf{The proofs of statements \eqref{dualoptimal:exp:s1} and \eqref{dualoptimal:exp:s2} of the Lemma.}
  Suppose that positive integers $k$ and $p$, coefficients $(\tilde{\delta}_j)_{j=1}^{k}$ and $(\tilde{\beta}_j)_{j=1}^{k}$,
  and the polynomial $\tilde{q}$ satisfy the conditions of the lemma.
  Let us consider the following primal--dual pair of LP problems.
  \begin{align}
    \text{Maximize } & 0 \text{ with respect to } (\delta_j)_{j=1}^{k}, (\beta_j)_{j=1}^{k}
      \text{ subject to \eqref{lp:primal:exp} with } r = \Cexp(k, p). \label{lp:primal:a} \\
    \text{Minimize } & y_0 \text{ with respect to } (y_m)_{m=0}^p \text{ subject to \eqref{dual:lp:exp:a}
      and \eqref{dual:lp:exp:b} with } r = \Cexp(k, p) . \label{lp:dual:a}
  \end{align}
  On the one hand $(\tilde{\delta}_j)_{j=1}^{k}$, $(\tilde{\beta}_j)_{j=1}^{k}$ is an optimal solution to
  \eqref{lp:primal:a}, and on the other hand $y_m := 0$ for $0 \le m \le p$ is a feasible solution to \eqref{lp:dual:a}.
  Therefore both \eqref{lp:primal:a} and \eqref{lp:dual:a} are feasible, and thus
  the Duality theorem in LP (Proposition ~\ref{prop:lp-duality}) implies that the two optima are equal,
  hence a feasible solution $(y_m)_{m=0}^p$ to \eqref{lp:dual:a} is optimal if and only if
  $y_0 = 0$.

  Let $\tilde{y}_m := \tilde{q}^{(m)}(0) / m!$ for $0 \le m \le p$ be the coefficients of $\tilde{q}$.
  Then $\tilde{y}$ is an optimal solution to \eqref{lp:dual:a},
  because $\sum_{m=0}^p \tilde{y}_m j^m = \tilde{q}(j) \ge 0$ for all $1 \le j \le k$,
  and $\sum_{m=0}^p (-mj^{m-1} + \Cexp(k,p) j^m)\tilde{y}_m = -\tilde{q}'(j) + \Cexp(k,p) \tilde{q}(j) \ge 0$ for all $1 \le j \le k$,
  and $\tilde{y}_0 = \tilde{q}(0) = 0$.

  Suppose that $\tilde{\delta}_j \ne 0$ for some $1 \le j \le k$. Then the Complementary
  slackness theorem (Proposition~\ref{prop:compslack}) implies that
  $\tilde{q}(j) = \sum_{m=0}^pj^m\tilde{y}_m = 0$. Similarly, if $\tilde{\beta}_j \ne 0$ for some $1\le j \le k$
  then the Complementary slackness theorem implies
  $-{\tilde{q}}'(j)+\Cexp(k,p) \tilde{q}(j) = 0$. Hence the proofs of statements
  \eqref{dualoptimal:exp:s1} and \eqref{dualoptimal:exp:s2} are complete.

  \textbf{The proof of statement \eqref{dualoptimal:exp:s3} of the Lemma.}
  Suppose that positive integers $k$ and $p$, coefficients $(\tilde{\delta}_j)_{j=1}^{k}$ and $(\tilde{\beta}_j)_{j=1}^{k}$,
  and the polynomial $\tilde{q}$ satisfy the conditions of the lemma.
  Suppose further that the vectors $a_j$ for $j \in \mathcal{I}$ and $b_j$ for
  $j \in \mathcal{J}$ are linearly independent.
  Consider the system of linear equations obtained from the equality
  constraints of the LP problem \eqref{lp:primal:a} by fixing
  $\delta_j$ for $j \not\in \mathcal{I}$ and $\beta_j$ for $j\not\in \mathcal{J}$
  to zero.
  In view of statements \eqref{dualoptimal:exp:s1} and \eqref{dualoptimal:exp:s2} of the lemma, this system is satisfied by
  the possibly non-zero coefficients
  $(\delta_j)_{j \in \mathcal{I}}$ and $(\beta_j)_{j\in \mathcal{J}}$
  of any optimal method.
  Since the matrix of this system is formed by the column vectors
  $a_j^\transp$ for $j\in \mathcal{I}$ and $b_j^\transp$ for $j \in \mathcal{J}$,
  its columns are linearly independent, and thus
  the Rouch\'e--Capelli theorem (Proposition \ref{prop:rouchecapelli}) implies
  that $(\tilde{\delta}_j)_{j \in \mathcal{I}}$, $(\tilde{\beta}_j)_{j\in \mathcal{J}}$
  is a unique solution to this system.
  From this, the third statement of the lemma follows.

  \textbf{The proof of statement \eqref{dualoptimal:exp:s4} of the Lemma.}
  Suppose that positive integers $k$ and $p$, coefficients $(\tilde{\delta}_j)_{j=1}^{k}$ and $(\tilde{\beta}_j)_{j=1}^{k}$,
  and the polynomial $\tilde{q}$ satisfy the conditions of the lemma.
  Suppose further that the polynomial $\tilde{q}$ is chosen so that it satisfies
  the maximum number of inequalities in \eqref{dual:poly2:exp:a}--\eqref{dual:poly2:exp:b}
  with equality.
  Let $\tilde{y}_m$ for $1 \le m \le p$ be as before.

  We show that $\tilde{y}$ satisfies at least one inequality in
  \eqref{dual:lp:exp:a}--\eqref{dual:lp:exp:b} with strict inequality.
  Suppose that is not the case, then for all $1 \le j \le k$ integers $j$ is
  a multiple root of $\tilde{q}$, and thus $\tilde{q}$ satisfies the system
  \eqref{dual:poly2:exp:a}--\eqref{dual:poly2:exp:c} with $r = 0$.
  Lemma~\ref{lem:dualcond:exp} leads to the contradiction $\Cexp(k,p) = 0$.

  Consider the homogeneous system of linear equations formed by the equation $y_0 = 0$ and by
  the equations corresponding to the binding inequalities in \eqref{dual:lp:exp:a}--\eqref{dual:lp:exp:b} with $r = \Cexp(k,p)$.
  The matrix of this system is formed by the row vectors
  $a_j$ for $j\in \mathcal{I}\cup\{0\}$ and $b_j$ for $j \in \mathcal{J}$.

  The rank of this matrix is at most $p$, since $\tilde{y}$ is a
  non-zero solution to the corresponding homogeneous system with $p+1$ variables.
  Suppose now that the rank of the matrix is less than $p$.
  Then, by the Rouch\'e--Capelli theorem (Proposition \ref{prop:rouchecapelli}),
  we have a $\hat{y}$ solution to the system such that
  $\hat{y}$ and $\tilde{y}$ are linearly independent, and
  $\hat{y}$ is an optimal solution to the LP problem \eqref{lp:dual:a}, too.
  Now, for appropriate real numbers $\tilde{c}$ and $\hat{c}$, the vector
  $\tilde{c}\cdot \tilde{y} + \hat{c}\cdot \hat{y}$ is a non-zero optimal
  solution to the LP problem \eqref{lp:dual:a}, and it satisfies more
  inequalities in \eqref{dual:lp:exp:a}--\eqref{dual:lp:exp:b} with equality
  than $\tilde{y}$ does. We can check that the corresponding polynomial
  $Q(x) := \sum_{m=0}^p(\tilde{c}\cdot \tilde{y}_m + \hat{c}\cdot \hat{y}_m)x^m$
  satisfies all the conditions of the lemma, and it satisfies more inequalities
  in \eqref{dual:poly2:exp:a}--\eqref{dual:poly2:exp:b} with equality than $\tilde{q}$
  does, which contradicts our choice of $\tilde{q}$.
  Hence the rank of the matrix is precisely $p$.

  Consider again the linear system for the possibly non-zero coefficients
  $(\delta_j)_{j \in \mathcal{I}}$ and $(\beta_j)_{j\in \mathcal{J}}$, obtained
  from the equality constraints of the LP problem \eqref{lp:primal:a}.
  The matrix of the system is formed by the column vectors
  $a_j^\transp$ for $j\in \mathcal{I}$ and $b_j^\transp$ for $j \in \mathcal{J}$,
  and the right hand side of the system is just $a_0^\transp$.
  Since this system is consistent, the Rouch\'e--Capelli theorem implies that
  the rank of its matrix is equal to the rank of its augmented matrix.
  We already know that the rank of its augmented matrix is precisely $p$,
  from which the proof of the statement follows.
\end{proof}

\begin{lem}
  \label{lem:dualoptimal:imp}
  Let positive integers $k$ and $p > 1$  be such that $\Cimp(k,p) > 0$.
  Assume that $(\delta_j)_{j=1}^{k}$ and $(\beta_j)_{j=0}^{k}$ are the coefficients of an
  implicit LM method of order $p$ with $k$ steps and with SSP coefficient $\Cimp(k, p)$,
  and let $q$ be any non-zero polynomial of degree at most $p$ that satisfies the system
  \eqref{dual:poly2:exp:a}--\eqref{dual:poly2:exp:d} with $r=\Cimp(k,p)$.
  Let us define the sets
  $\mathcal{I} := \{j : 1 \le j \le k, \text{ and } q(j) = 0\}$ and
  $\mathcal{J} := \{j : 0 \le j \le k, \text{ and } -q'(j)+\Cimp(k,p) q(j) = 0\}$,
  and let us introduce the vectors $a_j := (j^m)_{m=0}^{p}$ for $0 \le j \le k$,
  and $b_j := (-m j^{m-1}+\Cimp(k,p) j^m)_{m=0}^{p}$ for $1 \le j \le k$,
  and $b_0 = (0, -1, 0, \ldots, 0)$.
  Then the following statements hold.
  \begin{enumerate}[(i)]
    \item
      If $\delta_j \ne 0$ for some $1\le j\le k$, then
        $q(j) = 0$.
    \item
      If $\beta_j \ne 0$ for some $0\le j\le k$, then
      $-q'(j)+\Cimp(k,p) q(j) = 0$.
    \item
      If the vectors $a_j$ for $j \in \mathcal{I}$ and $b_j$ for $j \in \mathcal{J}$
      are linearly independent then there exists no other implicit LM method of
      order $p$ with $k$ steps and SSP coefficient $\Cimp(k,p)$.
    \item
      The polynomial $q$ can be chosen so that the vectors $a_j$ for $j \in \mathcal{I}$
      and $b_j$ for $j \in \mathcal{J}$ span a $p$ dimensional subspace of $\Real^{p+1}$.
  \end{enumerate}
\end{lem}
\begin{proof}
  The proof is essentially the same as that of Lemma~\ref{lem:dualoptimal:exp}.
\end{proof}

\section{Upper bounds on SSP coefficients}

In this section we derive upper bounds on SSP coefficients using Lemma~\ref{lem:dualcond:exp}
and Lemma~\ref{lem:dualcond:imp}.
Proposition~\ref{prop:upperbound:exp} is just the classical upper bound on the
SSP coefficient for explicit LM methods, found in \cite{Lenferink_1989}.

\begin{prop} \label{prop:upperbound:exp}
  Let $k, p$ be positive integers. Then the following inequality holds.
  \begin{equation*}
    \Cexp(k,p) \leq
      \begin{cases}
        \frac{k-p}{k-1} & \text{if $k \ge 2$ and $p \le k$} \\
        0 & \text{if $p > k$} \\
        1 & \text{if $k = p = 1$.}
      \end{cases}
  \end{equation*}
\end{prop}
\begin{proof}
  Let us consider $q(x) := x(k-x)^{p-1}$, a polynomial of degree $p$.
  If $k \ge 2$ then $q$ satisfies the system
  \eqref{dual:poly2:exp:a}--\eqref{dual:poly2:exp:c}
  with $r = \max\left(\frac{k-p}{k-1}, 0\right)$.
  If $k = 1$ and $p \ge 2$ then $q$ satisfies the same system with $r=0$.
  Finally, if $p = k = 1$ then $q$ satisfies the same system with $r=1$.
  The statement of the proposition follows from
  Lemma~\ref{lem:dualcond:exp}.
\end{proof}

Theorem \ref{thm:upperbound:imp} can be viewed as a refinement of the upper bound
$\Cimp(k,p) \leq 2 \text{ for all } k\geq 1 \text{ and } p\geq 2$ proved
in \cite{Lenferink_1991, Hundsdorfer_Ruuth_2005}. 

\begin{thm} \label{thm:upperbound:imp}
  Let $k \ge 1$ and $p \ge 2$ be arbitrary integers.
  Then the following inequality holds. 
  \begin{equation*}
    \Cimp(k,p) \leq
      \begin{cases}
          \frac{2k-p}{k-1} & \text{if $k \ge 2$ and $p \le 2k$} \\
        0 & \text{if $p > 2k$} \\
        2 & \text{if $k = 1$ and $p = 2$.}
      \end{cases}
  \end{equation*}
\end{thm}

\begin{proof}
  Let us consider $q(x) := x^2(k-x)^{p-2}$, a polynomial of degree $p$.
  If $k \ge 2$ then $q$ satisfies the system
  \eqref{dual:poly2:exp:a}--\eqref{dual:poly2:exp:d} with $r = \max\left(\frac{2k-p}{k-1}, 0\right)$.
  If $k = 1$ and $p \ge 3$ then $q$ satisfies the same system with $r=0$.
  Finally, if $k = 1$ and $p = 2$ then $q$ satisfies the same system with $r=2$.
  The statement of the theorem follows from Lemma~\ref{lem:dualcond:imp}.
\end{proof}

\begin{prop} \label{prop:optimal}
  The following statements hold.
  \begin{enumerate}[(i)]
    \item \label{prop:optimal:s1}
      Suppose that positive integers $p$ and $k$ are such that $\Cexp(k,p) > 0$.
      Then there exists an explicit LM method of order $p$ with $k$ steps, and with $\C = \Cexp(k, p)$.
    \item \label{prop:optimal:s2}
      Suppose that integers $p \ge 2$ and $k \ge 1$ are such that $\Cimp(k,p) > 0$.
      Then there exists an implicit LM method of order $p$ with $k$ steps, and with $\C = \Cimp(k, p)$.
  \end{enumerate}
\end{prop}
\begin{proof}
  The proof of statement \eqref{prop:optimal:s1} is based on Lemma~\ref{lem:farkas:exp} and Proposition~\ref{prop:upperbound:exp},
  and on the fact that if a polynomial $q$ solves the system \eqref{dual:poly:exp} with $r = r_0 > 0$,
  then the polynomial $\tilde{q} := q - q(0) / 2$ solves the same system with $r = r_1$ for some $0 < r_0 < r_1$.
  Similarly, the proof of statement \eqref{prop:optimal:s2} is based on Lemma~\ref{lem:farkas:imp} and Theorem~\ref{thm:upperbound:imp}.
\end{proof}

Theorem~\ref{thm:impexp} presents a relation between the optimal SSP coefficients for explicit and implicit LM methods
of different orders.
This inequality, together with the existence of arbitrary order implicit LM methods with $\C > 0$
proved in \cite{Sand_1986}, implies the existence of arbitrary order LM methods with $\C > 0$.
Using this approach, we can show that $\Cexp(k,p) > 0$ for a given order $p$, if the number of steps is at least $k \ge 2^{p-1}+1$.
In contrast, Theorem~\ref{thm:exist:exp} proves the existence of explicit LM methods with $\C > 0$
for a much lower number of steps.

\begin{thm} \label{thm:impexp}
  Let $k, p$ be positive integers. Then the following inequality holds.
  \begin{equation} \label{impexp}
    \Cimp(k,2p) \leq 2\Cexp(k,p).
  \end{equation}
\end{thm}
\begin{proof}
  Lemma \ref{lem:dualcond:exp} guarantees the existence of a non-zero real
  polynomial $q$ of degree at most $p$ that satisfies the system
  \eqref{dual:poly2:exp:a}--\eqref{dual:poly2:exp:c} with $r = \Cexp(k,p)$.
  Let us consider the polynomial $\tilde{q}:= q^2$.
  Clearly $\tilde{q}$ is a non-zero real polynomial of degree at most $2p$.
  We can check that $\tilde{q}$ satisfies conditions \eqref{dual:poly2:exp:a},
  \eqref{dual:poly2:exp:c} and \eqref{dual:poly2:exp:d}.

  The above $\tilde{q}$ satisfies condition \eqref{dual:poly2:exp:b}, too, with $r = 2\Cexp(k,p)$.
  This is because $-\tilde{q}'(j) + 2\Cexp(k,p)\tilde{q}(j) =
  2q(j)\left(-q'(j) + \Cexp(k,p)q(j)\right) \ge 0$ for all $1 \le j \le k$.
  Here we have used that $q$ satisfies conditions \eqref{dual:poly2:exp:a} and
  \eqref{dual:poly2:exp:b} with $r = \Cexp(k,p)$. 
  Hence $\tilde{q}$ satisfies the system \eqref{dual:poly2:exp:a}--\eqref{dual:poly2:exp:d}
  with $r = 2\Cexp(k,p)$, and thus Lemma~\ref{lem:dualcond:imp}
  implies the inequality \eqref{impexp}.
\end{proof}

\section{Existence of arbitrary order SSP explicit LMMs}
In \cite[Theorem 2.3(ii)]{Lenferink_1989}, it is asserted that there
exist explicit contractive linear multistep methods of arbitrarily high order;
the justification cited is \cite[Theorem~2.3]{Sand_1986}.  The latter Theorem
does prove existence of arbitrary-order contractive linear multistep methods;
however, it uses an assumption that if $\alpha_j\ne0$ for some $j$, then also
$\beta_j\ne 0$, which cannot hold for explicit methods (since necessarily 
$\alpha_0\ne 0$ and $\beta_0=0$).  Hence the methods constructed there are
necessarily implicit.

In view of Lemma \ref{lem:dualcond:exp}, the existence of arbitrary-order SSP explicit LMMs
can be shown by proving the infeasibility of \eqref{dual:poly2:exp:a}, \eqref{dual:poly2:exp:b} and
\eqref{dual:poly2:exp:c} with $r=0$ and a large enough $k$ for all $p$.
This can be achieved by applying Markov brothers-type inequalities to the polynomial $q$.
These inequalities give bounds on the maximum of the derivatives of a polynomial over an interval
in terms of the maximum of the polynomial, and
they are widely applied in approximation theory.

\begin{prop}[Markov brothers' inequality \cite{markovaa,markovva}] \label{prop:markov}
  Let $P$ be a real univariate polynomial of degree $n$. Then 
  \begin{equation*}
    \max_{-1 \le x \le 1} \abs{P^{(\ell)}(x)} \le
    \frac{n^2 (n^2 - 1^2) (n^2 - 2^2) \cdots (n^2 - (\ell-1)^2)}{1 \cdot 3 \cdot 5 \cdots (2\ell-1)} \max_{-1 \leq x \leq 1} \abs{P(x)}
  \end{equation*}
  for all $\ell$ positive integers. 
\end{prop}

\begin{thm} \label{thm:exist:exp}
  Let positive integers $k$ and $p$ be such that
  \begin{equation} \label{step:sufficient}
    k > \sqrt{\frac{p^2(p^2-1)}6\cdot\floorM{\frac{p+2}2}}.
  \end{equation}
  Then $\Cexp(k, p) > 0$, that is, there exists an explicit LM method with $k$ steps, order of accuracy $p$ and $\C > 0$.
\end{thm}
\begin{proof}
  Suppose to the contrary that $\Cexp(k, p) = 0$ for some $p$ and $k$ satisfying \eqref{step:sufficient}.
  Then by Lemma~\ref{lem:dualcond:exp}, there exists a non-zero real polynomial $q$ of degree at most $p$
  that satisfies the system \eqref{dual:poly2:exp:a}--\eqref{dual:poly2:exp:c} with $r = 0$.
  We choose this $q$ to satisfy properties \ref{prop1} and \ref{prop6} of Lemma \ref{lem:dualpropx:exp}.
  That is, $q$ is of degree $p$, and $q(x) \ge 0$ for all $x\in [0,k]$.

  First we define $a := \max_{x \in[0,k]} q(x)$ and $b := \max_{x\in[0,k]} |q''(x)|$,
  and we introduce the polynomial $P(x) := q((x+1)\frac{k}2) - \frac{a}{2}$.
  By construction, $\max_{x \in[-1,1]} |P(x)| = \frac12a$ and
  $\max_{x \in[-1,1]} |P''(x)| = \frac{k^2}{4}b$,
  and thus the Markov brothers' inequality for the second derivative
  (Proposition \ref{prop:markov} with $\ell = 2$) applied to $P$ implies that
  \begin{equation} \label{ba}
    b\cdot \frac{k^2}{4} \leq a\cdot\frac{p^2(p^2-1)}{6}.
  \end{equation}
  
  Now we give an upper estimate for $a$ in terms of $b$.
  By the definition of $b$, and since $q'(j) \ge 0$ for all integers $1 \le j \le k$
  due to condition \eqref{dual:poly2:exp:b},
  we have $q'(x) \le b(1-x)$ for $x \in [0,1]$, and 
  $q'(x) \le b\cdot \max(x-j, j+1-x)$ for $x \in [j, j+1]$ for all integers $1 \le j \le k-1$.
  From this, we get that $\int_0^1 q'(x) \diff x \le \frac{b}{2}$ and  
  $\int_j^{j+1} q'(x) \diff x \le \frac{b}{4}$ for all for all integers $1 \le j \le k-1$.

  Since $q'$ is a polynomial of degree $p-1$, it follows that 
  \begin{equation*}
    \left| \left\{j \in \mathbb{Z}: 0 \le j \le k-1, \text{ and } \max_{x\in[j,j+1]}q'(x) > 0\right\}\right| \le \left\lfloor{\frac{p}2} \right\rfloor,
  \end{equation*}
  that is, $q'$ can take positive values on at most $\left\lfloor\frac{p}{2}\right\rfloor$
  subintervals of $[0,k]$ of length one.
  Since $q(0) = 0$ and $q$ is non-negative on $[0, k]$, the interval $[0,1]$ is one of
  these subintervals.

  Now, because of $q(0) = 0$, we can estimate $a$ from above by
  \begin{equation} \label{ab}
    a \leq \int_0^k \max(q'(x), 0) \diff x \le \frac{b}{2} + \frac{b}{4}\cdot \left\lfloor\frac{p-2}{2}\right\rfloor
      = \frac{b}{4} \left\lfloor\frac{p+2}{2}\right\rfloor.
  \end{equation}

  Combining \eqref{ab} and \eqref{ba} yields that
  \begin{equation*}
    k \le \sqrt{\frac{p^2(p^2-1)}6\cdot\left\lfloor\frac{p+2}{2}\right\rfloor},
  \end{equation*}
  which contradicts the assumption of the theorem.
  Hence $\Cexp(k,p) > 0$ for all $k$ and $p$ positive integers satisfying
  \eqref{step:sufficient}.
\end{proof}

\section{Asymptotic behavior of the optimal step size coefficient for large number of steps}

Since for all positive $p$, $\Cexp(k,p)$ and $\Cimp(k,p)$ are trivially non-decreasing in $k$,
the following definitions are meaningful.
\begin{align*}
  \Cexp(\infty,p) & := \lim_{k\to \infty} \Cexp(k, p) &
  \Cimp(\infty,p) & := \lim_{k\to \infty} \Cimp(k, p)
\end{align*}
Due to the existence of arbitrary order explicit and implicit LM methods with $\C > 0$ (Theorem \ref{thm:exist:exp}),
and due to the upper bounds on SSP coefficients (Proposition~\ref{prop:upperbound:exp}
and Theorem~\ref{thm:upperbound:imp}), $0 < \Cimp(\infty,p) \le 1$ for all positive $p$,
and $0 < \Cimp(\infty,p) \le 2$ for all $p \ge 2$.

\begin{thm} \label{thm:asympt:exp}
  Let $p$ be a positive integer.
  Then the following statements hold.
  \begin{enumerate}[(i)]
    \item \label{asympt:exp:s1}
      If $p$ is odd then there exists a positive integer $K_p$ such that
      for all integers $k \ge K_p$
      \begin{equation} \label{asympt:odd:exp}
        \Cexp(k, p) = \Cexp(\infty, p).
      \end{equation}
    \item \label{asympt:exp:s2}
      If $p$ is even then there exists a positive integer $k_p$ such that
      for all integers $k > k_p$
      \begin{equation} \label{asympt:even:exp}
        \frac{1}{k-1} \le \Cexp(\infty, p-1) - \Cexp(k, p) \le \frac{1}{k-k_p}.
      \end{equation}
  \end{enumerate}
\end{thm}
\begin{proof}
  For positive integers $p$ and $n \le p$, let $L(p, n)$ denote the smallest whole number $L$ such that
  for all $k$ large enough that $\Cexp(k,p) > 0$, there exists a polynomial $q(x) := c\cdot \prod_{m=0}^p (x-\lambda_m)$
  of degree $p$ that satisfies the system \eqref{dual:poly2:exp:a}--\eqref{dual:poly2:exp:c} with $r = \Cexp(k,p)$,
  and whose roots satisfy
  \begin{equation} \label{roots:bound}
    \vert \{m : 1 \le m \le p, \Re(\lambda_n) \in [0, L]\} \vert \ge n;
  \end{equation}
  if no such $L$ exist, then we define $L(p,n) := \infty$.
  Let $N(p)$ denote the largest $n \le p$ positive integer such that $L(n, p) < \infty$.
  Because of properties \ref{prop1}--\ref{prop3} and \ref{prop5} of Lemma~\ref{lem:dualpropx:exp},
  $N(p)$ is odd for all $p \ge 1$.
  We break the proof into a number of steps.

  \textbf{Step 1.}
  Let integers $k \ge 2$ and $p \ge 1$ be such that $\Cexp(k,p) > 0$, and let $q$ be a polynomial
  that satisfies the system \eqref{dual:poly2:exp:a}--\eqref{dual:poly2:exp:c} with $r = \Cexp(k,p)$.
  We can check that $\tilde{q}(x) := (k-x)\cdot q(x)$, a polynomial of degree $p+1$,
  satisfies the same system with $r = \max\left(0, r_0-\frac{1}{k-1}\right)$,
  and thus Lemma~\ref{lem:dualcond:exp} implies that
  $\Cexp(k, p+1) \le \max\left(0, \Cexp(k, p) - \frac{1}{k-1}\right)$. 
  Hence $\Cexp(k,p)$ is strictly decreasing in $p$ until it reaches $0$.

  \textbf{Step 2.}
  Suppose that $N(p) = p$ for some positive integer $p$.
  We choose the positive integer $K_p$ so that $\Cexp(K_p, p) > 0$
  and $K_p \ge L(p,p) + p/\Cexp(K_p, p)$.
  Such a $K_p$ exists, because $\Cexp(k,p)$ is non-decreasing in $k$.
  Let $q$ be a polynomial of degree $p$
  that satisfies the system \eqref{dual:poly2:exp:a}--\eqref{dual:poly2:exp:c}
  with $r = \Cexp(K_p, p)$ and $k = K_p$, and whose roots
  satisfy \eqref{roots:bound} with $L = L(p, p)$ and $n = p$.
  We can check that $q$ satisfies the inequality \eqref{dual:poly2:exp:a} for all $j > K_p$.
  In view of statement \eqref{roots:s2} of Lemma~\ref{lem:dual:roots} and because of
  \begin{equation*}
    \sum_{m = 1}^p \frac{1}{j-\lambda_m} \le \sum_{m=1}^p\frac{1}{j-\Re(\lambda_m)} \le
    p\cdot \frac{1}{j-L(p,p)} \le \Cexp(K_p, p),
  \end{equation*}
  $q$ satisfies the inequality \eqref{dual:poly2:exp:b} with $r=\Cexp(K_p,p)$
  for all $j > K_p$.

  Therefore $q$ satisfies the system \eqref{dual:poly2:exp:a}--\eqref{dual:poly2:exp:c}
  with $r = \Cexp(K_p,p)$ for all $k\ge K_p$, and thus Lemma~\ref{lem:dualcond:exp} implies that $\Cexp(k,p) \le \Cexp(K_p,p)$
  for all $k \ge K_p$. Since $\Cexp(k,p)$ is non-decreasing in $k$, this means that
  $\Cexp(k,p) = \Cexp(K_p, p)$ for all $k \ge K_p$.
  Hence for all positive $p$, $N(p) = p$ implies that condition
  \eqref{asympt:odd:exp} holds for an appropriate choice of $K_p$.
 
  \textbf{Step 3.}
  Suppose that $N(p)=p$ holds for some positive integer $p$, and suppose that $q$ is
  a polynomial of degree $p$ that satisfies the system \eqref{dual:poly2:exp:a}--\eqref{dual:poly2:exp:c}
  with $r = \Cexp(\infty,p)$ for all positive $k$, and whose roots satisfy \eqref{roots:bound} with $L = L(p, p)$ and $n = p$.
  Such a $q$ exists in view of Step 2.
  Let us define $\mu := p/\Cexp(\infty, p)$, and let us introduce
  $\tilde{q}(x) := (x^2 - 2\left(4\mu+L(p,p)\right)\cdot x +
  \left(4\mu+L(p,p)\right)^2 + 4\mu^2) \cdot q(x)$, a polynomial of degree $p+2$.

  We can check that $\tilde{q}$ satisfies the inequality \eqref{dual:poly2:exp:a}
  for all positive $j$, and that $\tilde{q}$ satisfies condition \eqref{dual:poly2:exp:c}.
  Using the inequality $1/(j-\lambda)+1/(j-\conj{\lambda}) \le 1/\Im(\lambda)$,
  we can show, in a similar way as in Step 2, that $\tilde{q}$ satisfies the inequality
  \eqref{dual:poly2:exp:b} with $r = \frac34\Cexp(\infty, p)$ for all $j \ge L(p,p) + 2\mu$.
  Using that $1/(j-\lambda)+1/(j-\conj{\lambda})$ is decreasing in $j$
  in the range $0 \le j \le \Re(\lambda) - \Im(\lambda)$, we can show that
  $\tilde{q}$ satisfies the inequality \eqref{dual:poly2:exp:b} with
  $r = \Cexp(\infty, p) - 2(L(p,p)+4\mu) / ((L(p,p)+4\mu)^2 + 4\mu^2)$ for all
  $1 \le j < L(p,p) + 2\mu$.
  Therefore $\tilde{q}$ satisfies the system \eqref{dual:poly2:exp:a}--\eqref{dual:poly2:exp:c}
  with the maximum of the above expressions for $r$ for all positive $k$, and
  thus Lemma \ref{lem:dualcond:exp} implies that
  \begin{equation*}
    \Cexp(\infty,p+2) \le \max\left(\frac34\Cexp(\infty,p), \Cexp(\infty,p) - \frac{2(L(p,p)+4\mu)}{(L(p,p)+4\mu)^2 + 4\mu^2}\right).
  \end{equation*}
  Hence for all positive $p$, $N(p) = p$ implies that $\Cexp(\infty, p) > \Cexp(\infty,p+2)$.

  \textbf{Step 4.}
  Suppose that $N(p) = N(p_0) = p_0$ for some $p > p_0$ positive integers,
  and let $\epsilon > 0$ be arbitrary.
  In view of Step~2, there exists a positive integer $K_{p_0}$ such that
  $\Cexp(k,p_0) = \Cexp(\infty,p_0)$ for all $k \ge K_{p_0}$.
  Let us define $L := \max(L(p,p_0), K_{p_0})$ and $K := L + (p-p_0)/\epsilon$.

  Since $N(p) = p_0$, we can choose $k > K$ so that there
  exists a $q(x) := c\cdot\prod_{m=1}^p (x-\lambda_m)$ polynomial that
  satisfies the system \eqref{dual:poly2:exp:a}--\eqref{dual:poly2:exp:c}
  with $r = \Cexp(k, p)$, and whose roots satisfy the conditions
  $\Re(\lambda_1), \ldots, \Re(\lambda_{p_0}) \in [0,L]$ and
  $\Re(\lambda_{p_0+1}), \ldots, \Re(\lambda_p) \ge K$.
  Let us introduce $\tilde{q} := c\cdot (-1)^{p-p_0}\cdot\prod_{m=1}^{p_0} \left(x-\lambda_m\right)$,
  a polynomial of degree $p_0$.
  We can check that $\tilde{q}$ satisfies the conditions of statement \eqref{roots:s1} of
  Lemma~\ref{lem:dual:roots} for all positive $j$, and that $\tilde{q}$ satisfies
  condition \eqref{dual:poly2:exp:c}.
  Moreover, $\tilde{q}$ satisfies the conditions of statement \eqref{roots:s2} of Lemma~\ref{lem:dual:roots}
  with $r = \Cexp(k,p) + \epsilon$ for all $1\le j \le L$, since
  \begin{equation*}
    \sum_{m=0}^p\frac1{j-\lambda_m} = \sum_{m=0}^{p_0}\frac1{j-\lambda_m} +
    \sum_{m=p_0+1}^{p}\frac1{j-\lambda_m} \ge \sum_{m=0}^{p_0}\frac1{j-\lambda_m} +
    (p-p_0)\frac1{L-K} = \sum_{m=0}^{p_0}\frac1{j-\lambda_m} - \epsilon
  \end{equation*}
  for all $1 \le j \le L$.
  Now Lemma \ref{lem:dual:roots} implies that $\tilde{q}$ satisfies the system
  \eqref{dual:poly2:exp:a}--\eqref{dual:poly2:exp:c} with $r = \Cexp(k,p) + \epsilon$.
  From this, Lemma~\ref{lem:dualcond:exp} implies that
  $\Cexp(k,p) + \epsilon \ge \Cexp(L,p_0) = \Cexp(K_{p_0},p_0) = \Cexp(\infty,p_0)$.
  Since $\epsilon > 0$ is arbitrary and $\Cexp(\infty,p)$ is non-increasing in $p$,
  this means that $\Cexp(\infty, p) = \Cexp(\infty, p_0)$.
  Hence $N(p) = N(p_0) = p_0$ implies that $\Cexp(\infty, p) = \Cexp(\infty, p_0)$.

  Furthermore if $p = p_0 + 1$ then $p$ is even, and thus property \ref{prop4} of Lemma~\ref{lem:dualpropx:exp}
  guarantees that $\lambda_p = k$. Then a similar reasoning yields that $N(p) = N(p-1) = p-1$
  implies $\Cexp\left(\infty,p-1\right) - \Cexp(k,p) \le \frac{1}{k-L}$ for all $k > L$.

  \textbf{The proof of statement \eqref{asympt:exp:s1} of the Theorem.}
  Clearly, $N(1) = 1$ holds, due to condition \eqref{dual:poly2:exp:c}.

  Assume now that $N(2j+1) = 2j+1$ holds for all integers $0 \le j \le m-1$.
  Since $1 \le N(p) \le p$ and $N(p)$ is odd, we have $N\left(2m+1\right) = N(2j_0 + 1)$ for some integer $0 \le j_0 \le m$.
  Suppose that $j_0 < m$. Then Step 4 implies that $\Cexp(2m+1) = \Cexp(2j_0+1)$,
  which contradicts the result of Step 3.
  Therefore $N(2m + 1) = 2m +1$.

  By induction we conclude that $N(p) = p$ for all odd positive integers $p$.
  Finally, Step 2 implies statement \eqref{asympt:exp:s1} of the Theorem.
  
  \textbf{The proof of statement \eqref{asympt:exp:s2} of the Theorem.}
  Let $p$ be an arbitrary even positive integer.
  Then $N(p) = p_0$ for some odd positive integer $1 \le p_0 < p$, since $1 \le N(p) \le p$
  and $N(p)$ is odd.
  We already know that $N(p_0) = p_0$,
  thus Step~4 implies that $\Cexp(\infty, p) = \Cexp(\infty, p_0)$.
  Suppose that $p_0 \ne p-1$. Then Step~3 implies that $\Cexp(\infty,p) = \Cexp(\infty, p_0) > \Cexp(\infty, p-1)$,
  which contradicts the fact that $\Cexp(\infty, p)$ is non-increasing in $p$.
  Therefore $p_0 = p-1$, and thus Step~4 implies that
  $\Cexp\left(\infty,p-1\right) - \Cexp(k,p) \le \frac{1}{k-k_p}$ for a fixed $k_p$ and for all $k > k_p$.

  For $k$ large enough, the inequality $\frac{1}{k-1} \le \Cexp\left(\infty,p-1\right) - \Cexp(k,p)$
  follows from Step 1 and the fact that $\Cexp(k, p-1)$ is non-decreasing in $k$.
\end{proof}

\begin{thm} \label{thm:asympt:imp}
  Let $p \ge 2$ be an arbitrary integer.
  Then the following statements hold.
  \begin{enumerate}[(i)]
    \item
      If $p$ is even then there exists a positive integer $K_p$ such that
      for all integers $k \ge K_p$
      \begin{equation} \label{asympt:odd:imp}
        \Cimp(k, p) = \Cimp(\infty, p).
      \end{equation}
    \item
      If $p$ is odd then there exists a positive integer $k_p$ such that
      for all integers $k > k_p$
      \begin{equation} \label{asympt:even:imp}
        \frac{1}{k-1} \le \Cimp(\infty, p-1) - \Cimp(k, p) \le \frac{1}{k-k_p}.
      \end{equation}
  \end{enumerate}
\end{thm}
\begin{proof}
  The proof is essentially the same as that of Theorem~\ref{thm:asympt:exp}.
\end{proof}

The proofs of the following two theorems make use of the assumption that step
sizes are fixed.

\begin{thm} \label{thm:unique:exp}
  Let $p$ be an odd positive number, and let $K_p$ be the smallest positive
  integer such that $\Cexp(K_p, p) = \Cexp(\infty, p)$.
  Then for all $k \ge K_p$ there is a unique optimal explicit LM method with $k$
  steps, order of consistency $p$ and SSP coefficient $\Cexp(\infty, p)$,
  and for different values of $k$, the non-zero coefficients of the unique optimal
  methods coincide.
  Moreover, there is a finite sequence of positive integers $j_1, j_2, \ldots, j_{\frac{p-1}{2}}$,
  such that $1 < j_1 < j_1+1 < j_2 < j_2+1 < \ldots < j_{\frac{p-1}{2}} < j_{\frac{p-1}{2}}+1$,
  and that the coefficients of the optimal method satisfy $\delta_j=0$ for all $1 \le j \le k$,
  and that if $\beta_j \ne 0$ for some $1 \le j \le k$
  then $j \in \{1, j_1, j_1+1, \ldots, j_{\frac{p-1}{2}}, j_{\frac{p-1}{2}}+1\}$.
\end{thm}
\begin{proof}
  Suppose that $p$ and $K_p$ satisfy the conditions of the theorem.
  Theorem \ref{thm:asympt:exp} guarantees the existence of such a $K_p$ for all
  odd positive integers $p$.
  Let us set $\kappa_p := 2K_p$.
  Then $\Cexp(\kappa_p, p) = \Cexp(\infty, p) > 0$, and in view of Lemma \ref{lem:dualcond:exp}
  and property \ref{prop1} of Lemma \ref{lem:dualpropx:exp},
  there exists a $q$ polynomial of degree $p$ that satisfies the system
  \eqref{dual:poly2:exp:a}--\eqref{dual:poly2:exp:c} with $r = \Cexp(\infty, p)$
  and $k = \kappa_p$.

  First we show that $q$ satisfies the inequality \eqref{dual:poly2:exp:a} for all
  $1 \le j \le \kappa_p$ with strict inequality.
  Suppose to the contrary that $q(j_0) = 0$ for some positive integer $1 \le j_0 \le \kappa_p$.
  If $j_0 > K_p$ then $q$ satisfies the system \eqref{dual:poly2:exp:a}--\eqref{dual:poly2:exp:c}
  with $r = \Cexp(\infty, p) = \Cexp(K_p, p)$ and $k=K_p$, but does not
  satisfy property \ref{prop2} of Lemma \ref{lem:dualpropx:exp}, which is a contradiction.
  Similarly, if $j_0 \le K_p$ then the polynomial $\tilde{q}(x) := q(x-j_0)$
  satisfies the system \eqref{dual:poly2:exp:a}--\eqref{dual:poly2:exp:c}
  with $r = \Cexp(\infty, p) = \Cexp(K_p, p)$ and $k=K_p$, but fails to satisfy
  property \ref{prop2} of Lemma \ref{lem:dualpropx:exp}, which is again a contradiction.
  Hence none of the inequalities in \eqref{dual:poly2:exp:a} hold with equality.
  
  Since $-q'+\Cexp(\infty, p)q$ is a polynomial of degree $p$, it satisfies at
  most $p$ inequalities in \eqref{dual:poly2:exp:b} with equality.
  Because of statement \eqref{dualoptimal:exp:s4} of Lemma \ref{lem:dualoptimal:exp} and due to
  $q(j) > 0$ for all $1 \le j \le \kappa_p$, we can choose
  $q$ so that the number of binding inequalities in \eqref{dual:poly2:exp:b}
  is precisely $p$. Then $-q'(0)+\Cexp(\infty, p)q(0) = -q'(0) < 0$ because of
  properties \ref{prop5} and \ref{prop6} of Lemma \ref{lem:dualpropx:exp}.
  Let $\mathcal{J}$ and $b_j$ for $1 \le j \le \kappa_p$ be as in Lemma \ref{lem:dualoptimal:exp}.
  We can see that \eqref{dual:poly2:exp:b} holds for all $1 \le j \le \kappa_p$
  if and only if $\mathcal{J} = \{1, j_1, j_1+1, \ldots, j_{\frac{p-1}{2}}, j_{\frac{p-1}{2}}+1\}$
  for a sequence of integers $j_1, j_2, \ldots, j_{\frac{p-1}{2}}$
  satisfying $1 < j_1 < j_1+1 < \ldots < j_{\frac{p-1}{2}} < j_{\frac{p-1}{2}} + 1 \le \kappa_p$.
  Now property \ref{prop2} of Lemma \ref{lem:dualpropx:exp} implies that $q(j) > 0$ for
  all $j > \kappa_p$. We can see that $-q'(j)+\Cexp(\infty, p)q(j) > 0$ for all
  $j > \kappa_p$. This is because the set $\mathcal{J}$ contains all the $p$ roots of
  $-q'+\Cexp(\infty, p)q$.
  It follows that $q$ satisfies the system \eqref{dual:poly2:exp:a}--\eqref{dual:poly2:exp:c}
  with $r=\Cexp(\infty, p)$ for all positive $k$.

  Suppose that $\delta_j$ and $\beta_j$ for $1 \le j \le k$ are the coefficients
  of an explicit LM method with $k \ge K_p$ steps, order of consistency $p$, and SSP coefficient $\Cexp(\infty,p)$.
  Proposition \ref{prop:optimal} guarantees the existence of such a method for all $k \ge K_p$.
  Statement \eqref{dualoptimal:exp:s1} of Lemma \ref{lem:dualoptimal:exp} implies that $\delta_j = 0$ for all $1 \le j \le k$,
  and statement \eqref{dualoptimal:exp:s2} of Lemma \ref{lem:dualoptimal:exp} implies that if $\beta_j = 0$ for some
  $1 \le j \le k$, then $j \in \mathcal{J}$.
  Since $b_j$ for $j \in \mathcal{J}$ are linearly independent (by statement \eqref{dualoptimal:exp:s4} of Lemma \ref{lem:dualoptimal:exp} and our choice of $q$),
  statement \eqref{dualoptimal:exp:s3} of Lemma \ref{lem:dualoptimal:exp} implies that there is a unique optimal method
  with SSP coefficient $\Cexp(\infty,p)$ for all $k \ge K_p$.
  The possibly non-zero coefficients of this unique optimal method, namely $\beta_j$ for $j \in \mathcal{J}$,
  satisfy the same linear system for all $k \ge K_p$. Thus for different values of $k$,
  the non-zero coefficients of the unique optimal methods coincide.
\end{proof}
\begin{thm} \label{thm:unique:imp}
  Let $p$ be an even positive number, and let $K_p$ be the smallest positive
  integer such that $\Cimp(K_p, p) = \Cimp(\infty, p)$.
  Then for all $k \ge K_p$ there is a unique optimal implicit LM method with $k$
  steps, order of consistency $p$ and SSP coefficient $\Cimp(\infty, p)$,
  and for different values of $k$, the non-zero coefficients of the unique optimal
  methods coincide.
  Moreover, there is a finite sequence of positive integers $j_1, j_2, \ldots, j_{\frac{p-2}{2}}$,
  such that $1 < j_1 < j_1+1 < j_2 < j_2+1 < \ldots < j_{\frac{p-2}{2}} < j_{\frac{p-2}{2}}+1$,
  and that the coefficients of the optimal method satisfy $\delta_j=0$ for all $1 \le j \le k$,
  and that if $\beta_j \ne 0$ for some $0 \le j \le k$
  then $j \in \{0, 1, j_1, j_1+1, \ldots, j_{\frac{p-2}{2}}, j_{\frac{p-2}{2}}+1\}$.
\end{thm}
\begin{proof}
  The proof is essentially the same as for Theorem \ref{thm:unique:exp}.
\end{proof}

\bibliographystyle{alpha}
\bibliography{ssp-lmm}
\end{document}